\documentclass[14pt]{article}
\usepackage[cp1251]{inputenc}
\usepackage[english]{babel}
\usepackage{amsfonts,amssymb,amsmath,amstext,amsbsy,amsopn,amscd,amsthm,graphicx,euscript}
\usepackage{graphics}
\newtheorem{thm}{Theorem}
\newtheorem{lm}{Lemma} 
\newtheorem{crl}{Corollary} 
\newtheorem{dfn}{Definition} 
\newtheorem{rk}{Remark}

\def\:{\colon}

\def\0{{\mathbf 0}}
\def\1{{\mathbf 1}}

\def\R{{\mathbb R}}

\title{Cobordisms of graphs. A sliceness criterion for stably odd free knots and related results on cobordisms}
\author{D. Fedoseev, V. Manturov}

\begin{document}

\maketitle

\abstract{In \cite{FM}, the authors proved a sliceness criterion for {\em odd free knots}:
free knots with odd chords. In the present paper we give a similar criterion for stably odd free knots. 


In essence, free knots may be considered as framed 4-graphs. That leads to an important notion of 4-graph cobordism and the associated genera.

Some additional results on knot sliceness and cobordism are given.

}

\section{Introduction}

In the present paper, we consider framed $4$-graphs. Recall that a $4$-graph is {\em framed} if at each crossing
all four outgoing edges are split into two sets of {\em formally opposite}. 
By slightly abusing notation, we also admit circles without vertices as $4$-graphs or connected components of $4$-graphs.

Given a $4$-graph $\Gamma$, we are interested in {\em spanning surfaces} $S$ for $\Gamma$ which look as follows.
The graph $\Gamma$ itself is an image of the disjoint union $S^{1}\sqcup S^{1}\sqcup \cdots\sqcup S^{1}$ under a map $f$ which glues together some circles transversally at some pairs of points. We are looking for a $2$-surface $\Sigma$ of genus $g$ with boundary $S^{1}\sqcup S^{1}\sqcup \cdots\sqcup S^{1}$ such that the map $f$ can be extended to a map $f:\Sigma\to K$,
where $K$ is a $2$-complex obtained from $\Sigma$ by pasting some points {\em generically} (for more detail, see \cite{FM}).
We shall call $\Gamma$ the {\em boundary} of the complex $K$. 

We say that a framed $4$-graph $\Gamma$ has $k$ unicursal components if it can be obtained by pasting $k$ copies
of circles $S^{1}\sqcup \cdots \sqcup S^{1}$. Note that in general if the number of unicursal components of $\Gamma$ is greater then 1, the spanning surface $\Sigma$ is not required to be connected. Moreover, the number of connected components of $\Sigma$ must be equal to the number of unicursal components of $\Gamma$.

Here we distinguish between the following cases:

\begin{enumerate}
\item[]{\bf The general case:} We impose no restrictions on the complex $K$: it can have double lines, cusps, and triple points;
\item[]{\bf The case without cusps:} We require that the neigbourhood of each point of $K$ not belonging to $\Gamma$ 
looks locally like a union of one, two or three coordinate planes at the origin in $\R^{3}$;
\item[]{\bf The case without triple points:} We allow cusps, but forbid triple points;
\item[]{\bf The elementary case:} We forbid triple points and cusps.
\item[]{\bf The purified case:} We forbid triple points and cusps incident to a double line without an endpoint on the boundary of the complex $K$.
\end{enumerate}

\begin{dfn}
We shall call these complexes {\em standard complexes} (with boundary), (without cusps, without triple points
or elementary) respectively.

Analogously, one defines the corresponding {\em standard complexes} without boundary.  
\end{dfn}

\begin{dfn}
Let $K$ be a standard complex with boundary $\Gamma=\partial K$; then we can
distinguish between the following types of double lines of $K$:
\begin{enumerate}
\item {\em cyclic} double line;
\item {\em double line between two cusps};
\item {\em double line between two crossings of $\Gamma$};
\item {\em double line between a crossing of $\Gamma$ and a cusp}.
\end{enumerate}
Double lines of the first two types are called {\em interior}, other double lines are called {\em boundary}.
\end{dfn}

\begin{dfn} \label{genera}
The above four cases lead us to five types of spanning surfaces: {\em the general spanning surface},
{\em spanning surface without cusps}, {\em spanning surface without triple points},
{\em elementary spanning surfaces} and {\em purified spanning surface}.

When considering minimal surfaces of these five types, we get the following five genera of framed $4$-graphs:
$g(\Gamma), g_{0}(\Gamma), g'(\Gamma), g_{el}(\Gamma), g_{p}(\Gamma)$, the {\em slice genus}, the {\em slice genus without cusps},
the {\em slice genus without triple points}, the {\em elementary slice genus} and {\em the purified slice genus}, respectively. If the spanning surface is not connected, its genus is the total genus of its components.
\end{dfn}

Let us emphasise that when speaking about the genus of the complex $K$ we mean the genus of the 2-surface which is the preimage of $K$.

If a graph $\Gamma$ doesn't allow a spanning surface of a certain type of any genus, the corresponding genus of $\Gamma$ will be denoted by $\infty$. Such situation appears, for example, for $g_0$ and a graph with a single chord. The following statement follows directly from the definition:

\begin{thm}
For each framed $4$-graph $\Gamma$ we have $$\infty \ge g(\Gamma)\ge g_{0}(\Gamma)\ge g_{el}(\Gamma)\ge 0$$
and $$g(\Gamma)\ge g'(K)\ge g_{el}(\Gamma).$$
\end{thm}

We say that a framed $4$-graph $\Gamma$ is {\em slice} if $g(\Gamma)=0$. Respectively,
we shall use the terms {\em sliceness without cusps}, {\em sliceness without triple points}
and {\em elementary sliceness}.

\begin{rk}
Spanning surfaces are closely related to Reidemeister moves $\Omega_1 - \Omega_3$ for free knots with cusps being responsible 
for $\Omega_1-$moves and triple points being responsible for $\Omega_3-$moves.
The notion of spanning surfaces without triple points should be closely related with cobordisms of doodles, \cite{Fenn, FennKamada}.
\end{rk}

Note that the genus $g(\Gamma)$ is related not to a graph but rather to the corresponding free links, i.e.,
equivalence classes of framed $4$-graphs modulo Reidemeister moves,
\cite{ManturovParity, ManturovCobordism}.
Therefore Definition \ref{genera} introduces five important graph characteristics --- five sorts of genera.

Let us deal with $1$-component framed $4$-graphs. A standard procedure (see, e.g., \cite{ManturovParity})
associates with every such graph $\Gamma$ a chord diagram $C(\Gamma)$ where chords are in one-to-one correspondence with vertices of $\Gamma$. Two chords are {\em linked} if their endpoints alternate when passing along the core circle of the chord diagram.

\begin{dfn}
We say that a $1$-component framed $4$-graph $\Gamma$ is {\em odd} if each chord of it is linked with oddly many chords.
\end{dfn}

The main theorem from \cite{FM} is the following sliceness criterion:
\begin{thm}
For a $1$-component odd framed $4$-graph $\Gamma$ sliceness is equivalent to elementary sliceness:
$$g(K)=0\Longleftrightarrow g_{el}(\Gamma)=0.$$
\label{th1}
\end{thm}

By definition, both these conditions are equivalent to $g_{0}(\Gamma)=0$ and $g'(\Gamma)=0$.

This theorem is important for it reduces a ``dynamical'' problem of sliceness to a ``statical'' combinatorial problem which can be solved by a finite check of the graph $\Gamma$.

In the present paper, we investigate interrelations between various sliceness conditions and various genera

Recall, that there are two types of double lines on a complex $K$ with boundary $\Gamma$: interior and boundary. That naturally leads to three types of triple points of such a complex: {\em interior} triple points (incident to three interior double lines), {\em exterior} double lines (incident to three boundary double lines) and {\em mixed} triple points.

for framed $4$-graphs having one or many components.

In particular, we shall prove the following theorem.

\begin{thm}
Let $\Gamma$ be a $1$-component {\em iteratively odd} framed $4$-graph. It is slice if and only if it is slice with only exterior triple points present (that is, with no cusps, interior or mixed triple points).
\label{th2}
\end{thm}

Here by {\em iteratively odd} we mean the following. Consider a $1$-component framed $4$-graph $\Gamma$ and the corresponding
chord diagram $C(D)$. The chord deletion operation on $C(D)$ corresponds to ``ungluing'' operation for the corresponding
crossings. Hence, we can consider the graph $\Gamma^{1}=\Gamma'$ obtained by deleting all odd crossings of $\Gamma$.
By definition $\Gamma$ is odd if and only if $\Gamma'$ has is the circle without crossings. Some even crossings of
$\Gamma$ can become odd for $\Gamma'$. Hence, we can define $\Gamma^{2}= (\Gamma^{1})'$ and, iteratively, $\Gamma^{n}=(\Gamma^{n-1})'$.
We say that $\Gamma$ is {\em iteratively odd of order $k$} if $\Gamma^{k}$ has no crossings. $\Gamma$ is iteratively odd
if there exists a positive integer $k$ such that $\Gamma$ is iteratively odd of order $k$.

Respectively, we say that a crossing $c$ of $\Gamma$ has order $k$ if it disappears in $\Gamma^{k+1}$: odd crossings
have order $0$, etc. If $\Gamma$ is not iteratively odd and $\Gamma^{m}=\Gamma^{m+1}$ has nonempty set
of crossings, then all crossings of $\Gamma$ which persist of $\Gamma^{m}=\Gamma^{m+1}$ are called 
{\em stably even}.

This operation $\Gamma \to \Gamma'$ is called the {\em parity projection} (see \cite{ManturovParity, FMParity}).

The paper is organized as follows. We discuss the main ingredients of our proofs and prove them one-by-one.
The main ingredient is the {\em F-lemma} \ref{F_old} (see \cite{FM}), which states that if we have a complex $K$
of genus zero with some boundary $\Gamma$, then we can perform some ``smoothing operations'' at double lines,
cusps and triple points without touching the boundary leaving the genus of the complex unchanged.
The F-lemma is discussed in Section 2 of the paper. Section 3 is devoted to the proof of the Theorem \ref{th2} on the interatively odd knots. Section 4 deals with multi-component cases. In Section 5 we discuss the generalization \ref{F_general} of the F-lemma. In Section 6 further developments in the research area are presented.


\section{The F-lemma}

\begin{dfn}
We say that a standard complex $K'$ (with or without boundary) is a {\em smoothing} of a complex $K$ if 
the neighbourhood of $\partial K$ coincides with the neighbourhood of $\partial K$, and $K'$ is obtained by smoothings of some interior double lines of the complex $K$ (see \cite{FM})

We adopt the following notation: if $K'$ is a smoothing of $K$ we write $K'\prec K$; if $K'$ is obtained form $K$ by smoothing of exactly one double line we write $K'<K$.  
\end{dfn}

The F-lemma was given in \cite{FM} in the following way:

\begin{lm}[F-Lemma, \cite{FM}]
For each standard complex $K$ of genus $0$ 
with $\partial K=\Gamma$, there exists a {\em smoothing} $K'<K$ of genus $0$.
\label{F_old}
\end{lm}

The lemma can be proved via a regluing process. It allows one to prove Theorem \ref{th1}: starting with a standard complex $K$ of genus $0$
with $\partial K=\Gamma$, we undo the double lines of $K$ by smoothing and get a complex with a smaller number of
 $\#\{\mbox{interior lines}\}+\#\{\mbox{triple points}\}$; iterating this process,
we can get rid of all interior double lines. 

The parity used in the present paper is the {\em Gaussian parity} (see \cite{ManturovParity} for the initial definition and \cite{FMParity} for the definition in case of 2-knots). Note, that in the present paper we always mean free 2-knots, free surface knots and free 2- (or surface) links.

A simple observation shows that the following statement holds (see \cite{FMParity}):
\begin{lm} 
Let $\Gamma$ be framed 4-graph and $K$ be its spanning complex. Then if a double line connects a crossing $c$ of $\Gamma$ to 
a cusp, then $c$ is even.
\end{lm}

This means that all double lines of the resulting complex with boundary connect double points of $\Gamma$ to double points of $\Gamma$.



Another observation leads to the lemma:
\begin{lm}[\cite{FMParity}]
Let $K$ be a standard complex with boundary $\partial K=\Gamma$. Let $t$ be a triple point of $K$
such that all double lines are connected to crossings $c_{1},c_{2},c_{3}$, respectively. Then the sum of parities
of these crossings is even modulo $2$.
\label{fmpa}
\end{lm}

It follows from the fact that at every triple point the sum of the parities of the double lines, incident to it, equals $0 \mod 2$.

This means that if double lines originate from a triple point and come to some crossings, then either all 
of these crossings should be even or there should be one even and two odd crossings. 

In our situation when we have only odd crossings, such triple points are impossible. That completes
the proof of Theorem \ref{th1}.

The reason why Lemma \ref{fmpa} holds originates from the possibility to extend the Gaussian parity from
vertices of graphs to double lines of standard complexes of genus $0$. Note that in case of non-Gaussian parity that statement is not necessarily true.

\section{Parity for dimension $2$ and genus $0$} 

Now let us prove Theorem \ref{th2}.

To this end, we should recall the notion of {\em parity} and {\em parity projection} from \cite{ManturovCobordism,FMParity}.



A crucial property of this parity we shall need, is the following projection lemma: 
\begin{lm}
For a complex $K$ with $\partial K=\Gamma$, there is a complex $K'$ with 
$\partial K'=\Gamma'$ such that $K$ is obtained from $K'$ which is {\em smaller} than the complex $K$. 
\end{lm}

Here saying ``smaller'' we mean the following. Consider a {\em complexity} $\mathcal{C(K)}$ of a complex $K$ --- the triplet (\# triple points, \# cusps, \# double lines). Those triplets are ordered lexicographically. Then a complex $K_1$ is {\em smaller} than a complex $K_2$ if $\mathcal{C}(K_1)<\mathcal{C}(K_2)$.

In particular, if a cusp point $w$ of $K$ was connected to a crossing $c$ of $K$,
then this cusp point will persist in $K'=K^{1}, K^{2}, \dots$.

Iterating this process, we get
\begin{lm}
Let $\Gamma$ be a slice framed $4$-graph, $\Gamma=\partial K$. Then each crossing
$c$ of $\Gamma$ which is connected to a cusp is stably even.
\label{stably}
\end{lm}

Now we can prove Theorem \ref{th2}. Consider the slice iteratively odd graph $\Gamma$. Let $K$ be its spanning complex. Apply the F-lemma \ref{F_old} to the interior double lines of the spanning complex $K$. That gives us a spanning complex with equal or lower genus with no interior double lines and hence with no interior or mixed triple points. But it cannot have any cusps for all the interior cusps were deleted by the F-lemma and no cusps connected to boundary may exist due to Lemma \ref{stably}. That means that only exterior triple point may exist on such a complex. Theorem \ref{th2} is proved. \\

Theorem \ref{th2} has a simple

\begin{crl}
Consider a 1-component iteratively odd framed 4-graph $\Gamma$. It is elementary slice if and only if it is slice with no exterior triple points.
\end{crl}






\section{The case of many components}

It turns out that Theorem 1 holds in the case of multicomponent links if we treat parity in another sense.

Every framed $4$-graph $\Gamma$ having more than one component can be considered as the image
of {\em several circles} $S^{1}\sqcup \cdots \sqcup S^{1}$; we shall distinguish between vertices of $\Gamma$ 
which appear by pasting a component of one $S^{1}$ with itself and those which appear by
pasting different circles $S^{1}$. We refer to the former as {\em pure crossings} and to the latter as
{\em mixed crossings}.

\begin{thm}
Let $\Gamma$ be a $2$-component framed $4$-graph such that all crossings are mixed. Then
$\Gamma$ is slice if and only if it is elementary slice.
\label{thm3}
\end{thm}

The definition of sliceness for such multicomponent framed $4$-valent graphs uses 
standard $2$-complexes which are images of {\em several} $2$-discs (one for each unicursal component of the graph $\Gamma$) having standard intersections.



\begin{proof} To prove this theorem we use the following arguments. Let $\Gamma$ be a slice 2-component framed 4-graph and $K$ its spanning complex. The complex $K$ is an image the union of two 2-discs $D_1, D_2$. The interior double lines of $K$ are of the two types: those corresponding to the intersections between $D_1$ with $D_2$ (mixed double lines) and corresponding to the self-intersections between $D_1$ or $D_2$ (pure double lines).

Note that the F-lemma applies to 2-links. Indeed, smoothing a pure double line of a 2-link yields the same result as smoothing of a double line on a 2-knot and thus transforms the component of the link into a new complex of genus 0. To smooth a mixed double line we use the following trick: 
consider a connected sum of the two components.
That transforms the link into a 2-knot. Due to the F-lemma \ref{F_old} the smoothing of the double line on this complex produces a complex of genus 0. Now removing the added handle we obtain a 2-component link with the mixed double line smoothed.

As usual, the smoothing process can be applied to interior double lines of a link with boundary verbatim.

To prove Theorem \ref{thm3} we now apply this generalization of the F-lemma to the complex $K$. That gives us a new complex $K'$ with the same boundary $\partial K'= \Gamma$ and with no interior double lines. First we note, that since all crossings of $\Gamma$ are mixed, the complex $K'$ has no cusps for every cusp is a self-crossing of a component. Further, since there are only two components, there are no triple points: one cannot construct a triple point with all double lines incident to it being mixed given only two components of a link.

Therefore the complex $K'$ has neither cusps, nor triple points and Theorem \ref{thm3} is proved. 
\end{proof}

\section{Generalization of the F-lemma}

In this section we present some generalizations of the F-lemma. Lemma \ref{F_old} deals with a 2-knot --- that is a knotted sphere. In the previous section it was generalized to 2-links --- in other words, unions of knotted spheres. But it can be generalized further, to deal with knotted surfaces of arbitrary genus.

\begin{lm}[Generalized F-lemma]
\label{F_general}
Let $K$ be a surface knot of genus $g$ and $K'$ be a complex obtained from $K$ by smoothing a double line. Then the following holds: $$\chi(K)\le\chi(K').$$
\end{lm}

This lemma immediately gives the following

\begin{crl}
Let $K$ be a surface knot of genus $g$ and $K'$ be a complex obtained from $K$ by smoothing a double line. Let $K'$ be a connected oriented manifold. In that case, $K'$ is a surface knot of genus $g' \le g.$
\end{crl}

In other words, a smoothing can only decrease the genus of a knot. In particular, the original F-lemma follows directly from Lemma \ref{F_general}: for every double line of a 2-knot there is a smoothing producing a connected manifold. Its genus must be less than or equal to 0, and thus it is a sphere, in other words the resulting complex is a 2-knot.

That statement holds in case of surface links as well: the Euler characteristic of the link can only increase. In particular, if the number of link components is preserved by the smoothing, then the total genus of the link can only decrease. Moreover, we can give the following estimate.

Let $L$ be an $n-$component link and $L'$ be an $n'-$component link obtained from $L$ by smoothing of a double line. Let us denote the total genus of $L$ by $g$ and the total genus of $L'$ by $g'$. Then we have 

\begin{equation}
\label{genus_estimate}
g'\le g+(n'-n).
\end{equation}

The proof of that fact is trivial. Since all the link components are spheres with handles, we have $$\chi(L)=\sum_{i=1}^n (2-2g_i),$$ $$\chi(L')=\sum_{i=1}^{n'}(2-2g_i'),$$ where $g_i, g_i'$ denote the genera of the components of the links $L, L'$, respectively. Since $\chi(L)\le\chi(L')$, we obtain the necessary estimate.

Evidently, those statements are also true in case of knots and links with boundary if the smoothed double line is interior.




\section{Further developments}

In this section we present several further results dealing with some particular cases of knot sliceness and cobordism.

\begin{thm}
Let $\Gamma_1$ and $\Gamma_2$ be two framed 4-graphs. Assume there is a cobordism of genus $g$ between those graphs such that the corresponding complex $K, \; \partial K = \Gamma_1 \sqcup \Gamma_2,$ has no triple points and no double lines connecting cusps and points on the boundary. Then there exists an elementary cobordism of genus at most $g$ with the same boundary.
\end{thm}


\begin{proof}
Since the complex $K$ has no triple points, its double lines do not intersect. Moreover, all the cusps are incident to its interior double lines, and thus can be smoothed due to the generalized F-lemma without genus increase.
\end{proof}

Earlier we proved some properties of odd knots using Gaussian parity on 2-knots and 2-knots with boundary (see Theorem \ref{th1} and Theorem \ref{th2}). In general, Gaussian parity may not exist on surface knot due to different path on their surface diagrams being non-homotopic. Nevertheless, in some cases Gaussian parity still exists.

\begin{dfn}
A surface diagram on a cylinder $C$ is called {\rm normal} if it satisfies the following conditions:

\begin{enumerate}
\item If a curve $\gamma$ has its endpoints on the same boundary component of $C$, the curve $\gamma'$ paired with $\gamma$ has its endpoints on the same boundary component of $C$;
\item If paired curves $\gamma$ and $\gamma'$ have a common endpoint, and their other endpoints lie on the boundary of $C$, then they lie on the same boundary component of $C$;
\item If a curve $\gamma$ has endpoints on different boundary components of $C$, the curve $\gamma'$ paired with $\gamma$ has its endpoints also on different boundary components of $C$.
\end{enumerate}
\end{dfn}

In particular, those properties mean that every meridian of the cylinder $C$ not passing through curves endpoints intersects evenly many curves of the diagram. That means that Gaussian parity is correctly defined on such diagrams.

Now consider two genus zero cobordant free knots. The spanning surface for them is evidently a normal cylinder. Thus the following theorem holds:

\begin{thm}
Let $K_1$ and $K_2$ be two odd free knots. Then there exists a cobordism of genus 0 between $K_1$ and $K_2$ if and only if there exists an elementary genus 0 cobordism between them.
\end{thm}

\begin{proof}
The proof of this theorem follows the same pattern as the proof of Theorem \ref{th1}.

First we delete all interior double lines by means of the F-lemma.

Now all the double lines left have their endpoints on the boundary of the spanning surface and thus must be odd. Therefore, the spanning surface has neither cusps nor triple points.
\end{proof}

When dealing with smoothings (either of a crossing of a knot or of a double line of a surface knot) one has to possibilities: either the smoothing yields a connected object (a knot) or or not (a link). Let us call the smoothings of the first type ``good'' and of the second type --- ``bad''.

\begin{lm}
Let $K$ be a standard complex of genus 0 whose boundary $\partial K$ is a framed $4$-graph
$\partial K=\Gamma$. Assume there is a double line of $K$ connecting a vertex $x$ of $\partial K$
to a cusp point $p$ of $K$.

Consider a free knot $\Gamma'$ obtained from $\Gamma$ by the one of two smoothings at the vertex $x$, that gives a free knot (not a lin k). Then $\Gamma'$ is slice.

\end{lm}

\begin{proof}
Consider the complex $KK$ obtained by glueing two copies of $K$ along their boundary
$\partial K=\Gamma$. Since $K$ is a knotted disc, the complex $KK$ is a knotted sphere, i.e. a 2-knot, and the graph $\Gamma$ is a certain transversal section of the complex $KK$.

The complex $KK$ has a double line $\gamma$ which connects the two copies of the cusp $K$. There are two possible smoothing of this double line. One of them gives a connected complex which is a knotted sphere due to the F-lemma. 
Smoothing of a double line induces smoothing on a transversal section of the knot. It is easy to see, that the good smoothing of the double line $\gamma$ induces a one-component smoothing of the knot $\Gamma$ in the transversal section of $KK$. Indeed, two possible smoothings of $\Gamma$ correspond to two possible smoothings of the vertex $x$. The ``bad'' smoothing of $\Gamma$ separates $KK$ into two connected components and thus the transversal section is separated into two connected components as well because the smoothed double line intersects it transversally in one point. Therefore, the good smoothing of $\gamma$ can only correspond to the good smoothing of $x$. That completes the proof.


\end{proof}

Finally let us look at the ``bad'' smoothing of the vertex $x$ from the previous theorem. As it was shown, it corresponds to the ``bad'' smoothing of the double line $\gamma$ of the complex $KK$ which breaks it into a 2-component link. But due to the estimate (\ref{genus_estimate}) the total genus of those two components is at most 1. That means that at least one them is a knotted sphere. Thus, at least one component of the graph $\Gamma''$ obtained from $\Gamma$ by the ``bad'' smoothing at the vertex $x$ is capped with a disc, in other words is slice.

\end{document}